\documentclass[12pt]{amsart}

\usepackage{graphicx}
\usepackage{amssymb}

\usepackage{multido}
\usepackage[utf8]{inputenc}    % utf8 support       %!!!!!!!!!!!!!!!!!!!!
\usepackage[T1]{fontenc}
\usepackage{hyperref}

\hypersetup{
 colorlinks=true,       % false: boxed links; true: colored links
    linkcolor=cyan,          % color of internal links
    citecolor=green,        % color of links to bibliography
    filecolor=magenta,      % color of file links
   urlcolor=cyan           % color of external links
   }

\usepackage{shuffle}
\DeclareSymbolFont{bbold}{U}{bbold}{m}{n}
\DeclareSymbolFontAlphabet{\mathbbold}{bbold}
\newcommand{\ind}{\mathbbold{1}}

\allowdisplaybreaks
\numberwithin{equation}{section}
\numberwithin{figure}{section}

\newtheorem{thm}{Theorem}[section]
\newtheorem{cor}[thm]{Corollary}

\newtheorem{lem}[thm]{Lemma}
\newtheorem{prop}[thm]{Proposition}
\theoremstyle{definition}
\newtheorem{defn}[thm]{Definition}
\theoremstyle{remark}
\newtheorem{rem}[thm]{Remark}
\newtheorem{ex}[thm]{Example}

\newcommand{\bE}{\mathbb{E}}

\newcommand{\bI}{\mathbb{I}}
\newcommand{\bJ}{\mathbb{J}}
\newcommand{\bN}{\mathbb{N}}
\newcommand{\bP}{\mathbb{P}}
\newcommand{\bR}{\mathbb{R}}
\newcommand{\bW}{\mathbb{W}}

\newcommand{\cA}{\mathcal{A}}
\newcommand{\cP}{\mathcal{P}}
\newcommand{\cH}{\mathcal{H}}
\newcommand{\cS}{\mathcal{S}}
\newcommand{\cW}{\mathcal{W}}

\newcommand{\Tau}{T}

\newcommand{\PP}{\mathbb{P}}

\newcommand{\independent}{\bot \!\!\! \bot}

\allowdisplaybreaks
\begin{document}

\title[Growing random words]{Doob-Martin compactification of a Markov chain for growing random words sequentially}
\author[H.S. Choi]{Hye Soo Choi} 
\email{hyeso0oa@gmail.com}
 
\author[S.N. Evans]{Steven N. Evans}
\thanks{S.N.E. was supported in part by NSF grants DMS-0907639 and DMS-1512933, and NIH grant 1R01GM109454-01}
\address{Department of Statistics \#3860\\
 367 Evans Hall \\
 University of California \\
  Berkeley, CA  94720-3860 \\
   USA}
\email{evans@stat.berkeley.edu}

\begin{abstract}
We consider a Markov chain that iteratively generates a sequence of random finite words in such a way that the 
$n^{\mathrm{th}}$ word is uniformly distributed over the set of words of length $2n$ in which 
$n$ letters are $a$ and $n$ letters are $b$: at each step an $a$ and a $b$ are shuffled in 
uniformly at random among the letters of the current word.  We obtain a concrete characterization of the 
Doob-Martin boundary of this Markov chain and thereby delineate all the ways in which the Markov chain 
can be conditioned to behave at large times. Writing $N(u)$ for the number of letters $a$ (equivalently, $b$) 
in the finite word $u$, we show that a sequence $(u_n)_{n \in \bN}$ of finite words converges to a point in the 
boundary if, for an arbitrary word $v$, there is convergence as $n$ tends to infinity of the probability that 
the selection of $N(v)$ letters $a$ and $N(v)$ letters $b$ uniformly at random from $u_n$
and maintaining their relative order results in $v$.  
We exhibit a bijective correspondence between the points in the boundary and ergodic random total orders 
on the set $\{a_1, b_1, a_2, b_2, \ldots \}$ that have distributions which are separately invariant under 
finite permutations of the indices of the $a'$s and those of the $b'$s.  We establish a further bijective 
correspondence between the set of such random total orders and the set of pairs $(\mu,\nu)$ of 
diffuse probability measures on $[0,1]$ such that $\frac{1}{2}(\mu+\nu)$ is Lebesgue measure: the restriction of
the random total order to $\{a_1, b_1, \ldots, a_n, b_n\}$ is obtained by taking $X_1, \ldots, X_n$ (resp. $Y_1, \ldots, Y_n$) 
i.i.d. with common distribution $\mu$ (resp. $\nu$), letting
$(Z_1, \ldots, Z_{2n})$ be $\{X_1, Y_1, \ldots, X_n, Y_n\}$ in increasing order, and declaring that the $k^{\mathrm{th}}$ 
smallest element in the restricted total order is $a_i$ (resp. $b_j$) if $Z_k = X_i$
(resp. $Z_k = Y_j$).
\end{abstract}



% ---------------------------------------------------

\maketitle
% ----------------------------------------------------------------

\section{Introduction}

There is a very simple way of producing a 
uniformly distributed random permutation of a set with
$n$ objects, say $[n] := \{1,\ldots,n\}$:
we take the elements of $[n]$ in order and 
lay them down successively so that the $k^{\mathrm{th}}$ element
goes into a uniformly chosen one of the $k$ ``slots'' defined
by the $k-1$ elements that have already been laid down 
(the slot before the first element, the slot after the last element,
or one of the $k-2$ slots between elements).  This sequential
algorithm has the attractive feature that when the first
$k$ elements have been laid down they are in uniform random
order; that is, the algorithm builds uniformly distributed random permutations
of $[1], [2], \ldots, [n]$ in a sequential manner.

Suppose that we enumerate a standard deck of cards with the elements of the set
$[52]$.  If the deck is in some order, then the colors of the successive cards
({\bf R}ed or {\bf B}lack) define a word of length $52$ from the 
two-letter alphabet $\{R,B\}$ in which $26$ letters are $R$ 
and $26$ letters are $B$ (recall that a word of length
$k$ from a finite alphabet $\mathcal{A}$ is just
an element of the Cartesian product $\mathcal{A}^k$, although it
is usual to write the word $(a_1, \ldots, a_k)$ more succinctly
as $a_1 \cdots a_k$).  Moreover, if the order of the deck is 
random and uniformly distributed, then the resulting word is uniformly
distributed over the set of $\frac{52!}{26! 26!}$ such words.  

%Similarly,
%if the order of the deck is random and uniformly distributed,
%then the suits of the successive cards define a random word of length $52$ from the alphabet 
%$\{\spadesuit, \heartsuit, \diamondsuit, \clubsuit\}$
%that is uniformly distributed over the set of $\frac{52!}{13! 13! 13! 13!}$
%such words in which $13$ letters are $\spadesuit$, $13$ letters are $\heartsuit$,
%$13$ letters are $\diamondsuit$, and $13$ letters are $\clubsuit$.

Unfortunately, our sequential randomization algorithm doesn't have the feature
that at the $(2k)^{\mathrm th}$ step for $1 \le k \le 26$ we have a random
word from the alphabet $\{R,B\}$ that is uniformly distributed over
the set of $\binom{2k}{k}$ words in which $k$ letters are $R$ and $k$ letters
are $B$.  

However, there is a simple way of modifying our algorithm to
produce the latter type of random words sequentially.  
We begin at step $0$ with the
empty word.  Suppose that we have completed
$k$ steps and a word of length $2k$ has
been produced.  The first sub-step of step $k+1$
inserts the letter $R$
uniformly at random into one of the $2k+1$ slots defined by
these $2k$ letters to produce a word of length
$2k+1$.  The second sub-step inserts the letter
$B$ uniformly at random into one of the $2k+2$ slots defined
by these $2k+1$ letters to produce a word of length
$2k+2$ and thereby complete step $k+1$.  It is not difficult
to see that, despite the apparent dependence of this
procedure on the ordering of the letters $R$ and $B$,
this procedure does indeed achieve what it is claimed to achieve.

From now on we will replace the alphabet $\{R,B\}$ by the alphabet
$\{a,b\}$ and write $(U_n)_{n \in \bN_0}$ for the Markov
chain that arises from our random insertion procedure.  Thus,
$U_n \in \bW_n$, where $\bW_n$ is the set words drawn from the
alphabet $\{a,b\}$ that consist of $n$ letters $a$ and $n$ letters $b$.
Set $\bW := \bigsqcup_{n \in \bN_0} \bW_n$ and put $N(w) = n$ for
$w \in \bW_n$, $n \in \bN_0$.

 We investigate the infinite bridges (equivalently, the Doob $h$-transforms) 
 for the Markov chain $(U_n)_{n \in \bN_0}$; that is, the Markov chains that
 have the same backwards-in-time transition dynamics as $(U_n)_{n \in \bN_0}$.
 We thereby identify the Doob-Martin compactification
 of the state space $\bW$ of the Markov chain.  
This enables us to characterize the nonnegative harmonic functions for the
 Markov chain and hence delineate all the ways that the Markov chain can be conditioned to 
 ``behave at infinity''.  
 
 More specifically, we show that a $\bW$-valued
 Markov chain is an infinite bridge for the Markov chain $(U_n)_{n \in \bN_0}$ if and only if the
 backwards dynamics are given by removing one letter $a$ and one letter $b$ uniformly at random from the current word.  
We can enrich the state space of the Markov chain $(U_n)_{n \in \bN_0}$ by
 replacing $\bW_n$ with the set $\tilde \bW_n$ that consists of words made up from the letters
 $a_1, b_1, \ldots, a_n, b_n$ written down in some order (each letter appearing once); that is,
a word such as $aababb$ will be associated with a word such as $a_3 a_1 b_2 a_2 b_1 b_3$ -- a given
$w \in \bW_n$ has $(n!)^2$ associated words in $\tilde \bW_n$.
 We can then enhance an infinite bridge $(U_n^\infty)_{n \in \bN_0}$
 to produce a Markov chain $(\tilde U_n^\infty)_{n \in \bN_0}$ with values
 in $\tilde \bW := \bigsqcup_{n \in \bN_0} \tilde \bW_n$ such that given $U_n^\infty = u$
 the value of $\tilde U_n^\infty$ is uniformly distributed over all ways
 of ``subscripting'' the letters in $u$; for example, if
 $U_2^\infty = a b b a$, then $\tilde U_2^\infty$ is uniformly distributed over
 the four words $a_1 b_1 b_2 a_2$, $a_2 b_1 b_2 a_1$, $a_1 b_2 b_1 a_2$, 
 $a_2 b_2 b_1 a_1$.  Moreover, in going from $\tilde U_n^\infty$
 to $\tilde U_{n-1}^\infty$ the letters $a_n$ and $b_n$ are deleted.
 We may view $\tilde U_n^\infty$ as a random total (that is, linear)
 order on the set $\{a_1, b_1, \ldots, a_n, b_n\}$.  As $n$ varies, these
 orders are consistent in the sense that the order $\tilde U_n^\infty$
 induces on $\{a_1, b_1, \ldots, a_{n-1}, b_{n-1}\}$ is just the order
 given by $\tilde U_{n-1}^\infty$.  Consequently, there is a
 total order on $\{a_1, b_1, a_2, b_2, \ldots\}$ that induces each of
 the orders given by the $\tilde U_n^\infty$.  This total order is
 exchangeable in the sense that finite permutations of the subscripts of the
 $a$'s and $b$'s separately leave its distribution unchanged.  The
 infinite bridge $(U_n^\infty)_{n \in \bN_0}$ is extremal (that is, not
 a mixture of infinite bridges or, equivalently,
 has an almost surely trivial tail $\sigma$-field) if and only if the exchangeable random total order
 on  $\{a_1, b_1, a_2, b_2, \ldots\}$ is ergodic in the sense that
 if an event is unchanged by finite permutations of the subscripts of the
 $a$'s and $b$'s separately, then it has probability zero or one.
By general Doob--Martin theory, extremal bridges correspond to extremal
elements of the Doob--Martin boundary and, in general, some elements
of the Doob--Martin boundary may not be extremal.  We show that the latter phenomenon
does not occur in our setting -- all Doob--Martin boundary points are extremal.
 
 We demonstrate that there is a bijective correspondence 
 between ergodic exchangeable random total orders
 on  $\{a_1, b_1, a_2, b_2, \ldots\}$ and
 pairs $(\mu, \nu)$ of diffuse probability measures on the unit interval $[0,1]$ 
 such that $\frac{\mu + \nu}{2} = \lambda$, where $\lambda$ is Lebesgue measure on $[0,1]$:
let $V_1, V_2, \ldots$ be i.i.d. with distribution $\mu$ and $W_1, W_2, \ldots$
be independent and i.i.d. with distribution $\nu$, then, writing $\prec$
for the total order, we have $a_i \prec a_j$ (resp. $a_i \prec b_j$,
$b_i \prec a_j$, $b_i \prec b_j$) if $V_i < V_j$ 
(resp. $V_i < W_j$, $W_i < V_j$, $W_i < W_j$).  
Another way of describing this construction is the following. We only need to describe the
restriction of the random total order to $\{a_1, b_1, \ldots, a_n, b_n\}$ for each $n \in \bN_0$.  Let 
$(Z_1, \ldots, Z_{2n})$ be $\{V_1, W_1, \ldots, V_n, W_n\}$ in increasing order and declare that the $k^{\mathrm{th}}$ 
smallest element of $\{a_1, b_1, \ldots, a_n, b_n\}$  in the restricted total order is $a_i$ (resp. $b_j$) if $Z_k = X_i$
(resp. $Z_k = Y_j$).

We remark that, due to the relationship $\frac{\mu + \nu}{2} = \lambda$, the probability measure
$\nu$ is uniquely determined by the probability measure $\mu$ and {\em vice versa} and hence we could
have said that the ergodic exchangeable random total orders are in bijective correspondence with the
probability measures $\mu$ on $[0,1]$ that satisfy $\mu \le 2 \lambda$.  However, we find
the more symmetric description to be preferable.
 
%We develop all of our results in the case of the chain $U_0, U_1, \ldots$
%that has as its state space words with letters from the alphabet $\{a,b\}$.
%This is just for ease of notation.  Analogous results hold for the shuffling
%chains with larger alphabets discussed at the beginning of this Introduction. 
 
In terms of the Doob--Martin topology, we show that a sequence $(y_k)_{k \in \bN}$ 
with $y_k \in \bW_{N(y_k)}$ and $N(y_k) \to \infty$ as $k \to \infty$ converges
to the point in the Doob--Martin boundary corresponding to the pair of measures
$(\mu,\nu)$ if and only if for each $m \in \bN$ the the sequence of
random words obtained by
selecting $m$ letters $a$ and $m$ letters $b$ uniformly at random from
$y_k$ and maintaining their relative order converges in distribution as
$k \to \infty$ to the random word that is obtained by writing
$V_1, \ldots, V_m, W_1, \ldots, W_m$ in increasing order to make
a list $(Z_1, \ldots, Z_{2m})$ as above and then putting
a letter $a$ (resp. $b$) in position $\ell$ of the word when $Z_\ell \in \{V_1, \ldots, V_m\}$
(resp. $Z_\ell \in \{W_1, \ldots, W_m\}$).  Moreover, the convergence of $(y_k)_{k \in \bN}$ 
to $y$ is equivalent to the weak convergence of $\mu_k$ to $\mu$ and $\nu_k$ to $\nu$,
where $\mu_k$ (resp. $\nu_k$) is the probability measure that places mass $\frac{1}{N(y_k)}$
at the point $\frac{\ell}{2 N(y_k)}$ $1 \le \ell \le 2N(y_k)$, if the $\ell^{\mathrm{th}}$ letter of the word $y_k$ is
the letter $a$ (resp. $b$).

\section{Background on the Doob--Martin compactification}
\label{DM_background}

 The primary reference on the Doob--Martin 
compactification theory for discrete time Markov chains 
is \cite{MR0107098}, but useful reviews
may be found in 
\cite[Chapter 10]{MR0407981},
\cite[Chapter 7]{MR0415773}, 
\cite{MR1463727},
\cite[Chapter IV]{MR1743100},
\cite[Chapter III]{MR1796539}.
We restrict the following sketch to the setting that is
of interest to us.

Suppose that  $(X_n)_{n \in \bN_0}$ is a discrete time Markov chain
with countable state space $E$ and transition matrix $P$.  
Suppose in addition that $E$ can be partitioned
as $E = \bigsqcup_{n \in \bN_0} E_n$, where
$E_0 = \{e\}$ for some distinguished state $e$,
each set $E_n$, $n \in \bN_0$ is finite, and the transition matrix
$P$ is such that $P(k,\ell) = 0$ unless $k \in E_n$ and 
$\ell \in E_{n+1}$ for some $n \in \bN_0$.
Define the {\em Green kernel} or {\em potential kernel} $G$ of
$P$ by 
\[G(i,j) := \sum_{n=0}^\infty P^n(i,j) 
= \bP^i\{X_n = j \; \text{for some $n \in \bN_0$}\}
=:\bP^i\{\text{$X$  hits $j$}\},
\] 
$i,j \in E$,
and assume that $G(e,j) > 0$ for all $j \in E$, so
that any state can be reached with positive probability starting from $e$. 

The {\em Doob--Martin kernel with reference state $e$} is
\[ 
K(i,j) := \frac{G(i,j)}{G(e,j)} = 
\frac{\bP^i\{\text{$X$  hits $j$}\}}{\bP^e\{\text{$X$  hits $j$}\}}.
\] 
If $j,k \in E$ with $j \ne k$,
then $K(\cdot,j) \ne K(\cdot,k)$ and so $E$ can be identified
with the collection of functions $(K(\cdot,j))_{j \in E}$.  
Note that
\[
0 \le K(i,j) \le \frac{1}{\bP^e\{\text{$X$  hits $i$}\}},
\]
and so the set of functions $(K(\cdot,j))_{j \in E}$
is a pre-compact subset of $\bR_+^E$. Its closure $\bar E$
is the {\em Doob--Martin compactification} of $E$. 
The set $\partial E := \bar E \setminus E$
is the {\em Doob--Martin boundary} of $E$.

By definition, a sequence $(j_n)_{n \in \bN}$ in $E$ converges
to a point in $\bar E$
if and only if the sequence of real numbers
$(K(i,j_n))_{n\in\bN}$ converges for all $i\in E$. Each function
$K(i,\cdot)$ extends continuously to $\bar E$. The
resulting function $K: E \times \bar E \rightarrow \bR$ is
the {\em extended Martin kernel}.  For $y \in \partial E$ the
nonnegative function $K(\cdot,y)$ is harmonic and 
any nonnegative harmonic function can be represented as
$\int K(\cdot, y) \, \mu(dy)$ 
for a suitable finite measure $\mu$ on $\partial E$.

If $Z$ is a $\bP^e$-a.s. bounded random variable 
that is measurable with respect to the tail $\sigma$-field of
$(X_n)_{n \in \bN_0}$, then
$\bE^e[Z \, | \, X_0, \ldots, X_n] = h(X_n)$ for some bounded harmonic function
$h$ and,  by the martingale convergence theorem,
$\lim_{n \to \infty} h(X_n) = Z\,$ $\bP^e$-a.s.  
Conversely, if $h$ is a bounded harmonic function, then
$\lim_{n \to \infty} h(X_n)$ exists $\bP^e$-a.s. 
and the limit random variable is $\bP^e$-a.s. equal to a random variable
that is measurable with respect to the tail $\sigma$-field of 
$(X_n)_{n \in \bN_0}$. 

The limit 
$X_\infty:=\lim_{n \rightarrow \infty} X_n$ exists $\bP^e$-almost 
surely in the topology of $\bar E$ and the limit belongs to $\partial E$ 
$\bP^e$-almost surely.
The tail $\sigma$-field of 
$(X_n)_{n \in \bN_0}$ coincides $\bP^e$-almost surely
with the $\sigma$-field generated by $X_\infty$.

Each $j \in E = \bigsqcup_{n \in \bN_0} E_n$
belongs to a unique $E_n$ whose index $n$ we denote by $N(j)$.  If the Markov chain
starts in state $e$, then $N(j)$ is the only time that there is
positive probability the Markov chain will be in state $j$. Write 
$(X_0^j, \ldots, X_{N(j)}^j)$ for the {\em bridge} obtained by
starting the Markov chain in state $e$ and conditioning it to be in
state $j$ at time $N(j)$.  This process 
is a Markov chain with transition probabilities
\[
\begin{split}
\bP\{X_{n+1}^j = i'' \, | \, X_n^j = i'\}
& =
\frac{
\bP^e\{X_n = i', \, X_{n+1} = i'', \, X_{N(j)} = j\}
}
{
\bP^e\{X_n = i', \, X_{N(j)} = j\}
} \\
& =
\frac{
\bP^e\{ \text{$X$  hits $i'$}\} P(i',i'') \bP^{i''}\{ \text{$X$  hits $j$}\}
}
{
\bP^e\{ \text{$X$  hits $i'$}\} \bP^{i'}\{ \text{$X$  hits $j$}\}
} \\
& =
\frac{
P(i',i'') \bP^{i''}\{ \text{$X$  hits $j$}\} / \bP^e\{\text{$X$  hits $j$}\}
}
{
\bP^{i'}\{ \text{$X$  hits $j$}\} / \bP^e\{\text{$X$  hits $j$}\}
} \\
& =
K(i',j)^{-1} P(i',i'') K(i'',j). \\
\end{split}
\]
The backward transition probabilities of $(X_0^j, \ldots, X_{N(j)}^j)$
are given by
\[
\begin{split}
\bP\{X_n^j = i' \, | \, X_{n+1}^j = i''\}
& =
\frac{
\bP^e\{ \text{$X$  hits $i'$}\} P(i',i'') \bP^{i''}\{ \text{$X$  hits $j$}\}
}
{
\bP^e\{ \text{$X$  hits $i''$}\} \bP^{i''}\{ \text{$X$  hits $j$}\}
} \\
& =
\frac{
\bP^e\{ \text{$X$  hits $i'$}\} P(i',i'')
}
{
\bP^e\{ \text{$X$  hits $i''$}\}
}, \\
\end{split}
\]
so that all bridges have the same backward transition probabilities.
An {\em infinite bridge} for $(X_n)_{n \in \bN_0}$ 
is a Markov chain $(X_n^\infty)_{n \in \bN_0}$ with these backward transition
probabilities.  If $(X_n^\infty)_{n \in \bN_0}$ is an infinite bridge, then
\[
\begin{split}
\bP\{X_{n+1}^\infty = i'' \, | \, X_n^\infty = i'\}
& =
\frac{
\bP^e\{ \text{$X^\infty$  hits $i''$}\}
\bP\{X_n^\infty = i' \, | \, X_{n+1}^\infty = i''\}
}
{
\bP^e\{ \text{$X^\infty$  hits $i'$}\}
} \\
& =
h(i')^{-1} P(i',i'') h(i''), \\
\end{split}
\]
where
\[
h(i) = 
\frac{
\bP^e\{  \text{$X^\infty$  hits $i$}\}
}
{
\bP^e\{  \text{$X$  hits $i$}\}
}.
\]
Thus an infinite bridge is a Doob $h$-transform of $(X_n)_{n \in \bN_0}$
with a particular harmonic function $h$.  Conversely, any
Doob $h$-transform is an infinite bridge.

Suppose now that $(j_k)_{k \in \bN}$ is a sequence
of elements of the state space $E$ such that $N(j_k) \to \infty$
as $k \to \infty$.  As observed in \cite{MR0426176}, such a sequence
$(j_k)_{k \in \bN}$ converges in the Doob--Martin topology if and
only if finite initial segments of the corresponding bridges
converge in distribution.  Moreover, two
sequences of states converge to the same limit if and only if
the limiting distributions of finite
initial segments are the same.  For a sequence
$(j_k)_{k \in \bN}$ that converges to a point in the Doob--Martin boundary,
the limiting distributions of the 
initial segments define the distribution of an $E$-valued
Markov chain $(X_n^{(h)})_{n \in \bN_0}$ with
transition probabilities $P^{(h)}$ given by 
\[
P^{(h)}(i,j) := h(i)^{-1} P(i,j) h(j), \quad i,j \in E^{(h)},
\]
where $h(i) = \lim_{k \to \infty} K(i,j_k)$ 
and 
\[
\begin{split}
E^{(h)} 
& := \{i \in E : h(i) > 0\} \\
& = \{i \in E: \lim_{k \to \infty} \bP\{X_{N(i)} = i \, | \, X_{N(j_k)} = j_k\} > 0\}. \\
\end{split}
\]
This Markov chain $(X_n^{(h)})_{n \in \bN_0}$ is an infinite bridge.  A necessary
condition for an infinite bridge to be extremal (that is, having a distribution 
that is not a nontrivial mixture of infinite bridge distributions) 
is that it is of this form.

\section{Transition probabilities and the Doob--Martin kernel for the growing word chain}

\begin{defn} For $n \in \bN_0$ write $\bW_n$ for the set of words
from the alphabet $\{a,b\}$ that have $n$ letters $a$
and $n$ letters $b$ and put $\bW := \bigsqcup_{n \in \bN_0} \bW_n$.  
\end{defn}

By definition, the Markov chain $(U_n)_{n \in \bN_0}$ has state space
$\bW$ and one-step transition probabilities
\[
\bP\{U_{m+1} = w \, \vert \, U_m = v\}
=
\frac{M(v,w)}{(2m+2)(2m+1)} 
\]
for $v \in \bW_n$ and $w \in \bW_{n+1}$,
where $M(v,w)$ is the number of ways to write $w = v_1 x v_2 y v_3$ 
in such a way that
$\{x,y\} = \{a,b\}$ and $v_1, v_2, v_3$ are (possibly empty) words such that
$v = v_1 v_2 v_3$.  That is, $M(v,w)$ is the number of times that  
$v$  appears inside  $w$  as a {\em sub-word}. (We recall that, in general, a word
$c_1 \cdots c_p$ is a sub-word of a word $d_1 \cdots d_q$ if
there is a map $f:[p] \to [q]$ such that $f(i) < f(j)$ for $1 \le i < j \le p$ 
and $d_{f(k)} = c_k$ for $1 \le k \le p$.)  

In order to write down multi-step
transition probabilities for the Markov chain $(U_n)_{n \in \bN_0}$, it is convenient to introduce
the following standard notation (see, for example, \cite{MR1475463}).

\begin{defn}
Given two words  $w$   and   $v$ drawn from some finite alphabet,  
write $\binom{w}{v}$ for the number of times that $v$  appears  
as a sub-word of $w$. 
\end{defn}

\begin{ex}
For example, $\binom{abbaba}{bba}=4$ 
because $bba$ 
appears inside $abbaba$ as a sub-word four times:

\vspace{3mm}
a{\Large \textbf{bba}}ba 
\hspace{3mm} a\textbf{\Large bb}ab\textbf{\Large a} 
\hspace{3mm} a\textbf{\Large b}ba\textbf{\Large ba} 
\hspace{3mm} ab\textbf{\Large b}a\textbf{\Large ba}.
\vspace{3mm}
%Similarly,$\binom{abcbca}{bc}=3$,  because $bc$ 
%appears inside $abcbca$ as a sub-word three times:
 %
%\vspace{3mm}
%a{\Large \textbf{bc}}bca 
%\hspace{3mm} a\textbf{\Large b}cb\textbf{\Large c} a 
%\hspace{3mm} abc\textbf{\Large bc}a.
\end{ex}

\begin{rem}
Note that if our alphabet has only one letter, then $\binom{w}{v}$ is just
the usual binomial coefficient $\binom{|w|}{|v|}$, where we use the notation
$|u|$ for the length of the word $u$.
\end{rem}

For a general finite alphabet $\cA$, $\binom{w}{v}$ is uniquely determined by the 
following three properties, where we write $\cA^*$ for the set
of finite words with letters drawn from the alphabet $\cA$
(see \cite[Proposition 6.3.3]{MR1475463}):
\begin{itemize}
\item
$\binom{w}{\emptyset} = 1$ for all $w \in \cA^*$, where $\emptyset$ is the empty word,
\item
$\binom{w}{v} = 0$ for all $v,w \in \cA^*$ with $|w| < |v|$, 
\item
$\binom{w y}{v x} = \binom{w}{v x} + \delta_{x,y} \binom{w}{v}$, for all $v,w \in \cA^*$
and $x,y \in \cA$, where $\delta$ is the usual Kronecker delta.
\end{itemize}

The counting involved in determining $\binom{w}{v}$ for general $v,w \in \cA^*$ is handled by the
following result from \cite{Claesson_15}.  Define an infinite matrix $\cP$ with entries
indexed by $\cA^*$ by setting the $(v,w)$ entry to be $\binom{w}{v}$.  
If the row and column indices are ordered
so that they are nondecreasing in word length, 
then $\cP$ is an upper triangular matrix with $1$ in every position
on the diagonal.  
Define another infinite matrix $\cH$ indexed by $\cA^*$ by setting the $(v,w)$ entry to be
$\binom{w}{v}$ if $|w|=|v|+1$ and $0$ otherwise.  With the same ordering of the indices as for $\cP$,
the matrix $\cH$ is upper triangular with $0$ in every position on the diagonal.  
The matrix exponential $\exp(\cH)$ is well-defined and is equal to $\cP$.

Using the above notation, we can express the transition probabilities of 
$(U_n)_{n \in \bN_0}$ as follows.

\begin{lem}\label{2tran}
For words $v \in \bW_m$ and $w \in \bW_{m+n}$
	\begin{equation*}
	\PP\{U_{m+n}= w \, \vert \, U_m= v\} 
	= 
	\binom{w}{v} \frac{n!n!}{(2m+1) (2m+2) \cdots (2(m+n))}.
	\end{equation*}
\end{lem}

\begin{proof}
We proceed by induction.  The result is certainly
true when $n=1$.  Supposing it is true for some value of $n$, in order to show
it is true for $n+1$, we need to show that for $u \in \bW_m$ and $w \in \bW_{m+n+1}$
we have
\[
\begin{split}
& \sum_{v \in \bW_{m+1}}
\binom{v}{u} \frac{1}{(2m+1) (2m+2)}
\binom{w}{v} \frac{n!n!}{(2m+3) (2m+4) \cdots (2(m+n+1))} \\
& \quad =
\binom{w}{u} \frac{(n+1)!(n+1)!}{(2m+1) (2m+2) \cdots (2(m+n+1))}, \\
\end{split}
\]
or, equivalently, that
\[
\sum_{v \in \bW_{m+1}}
\binom{v}{u}
\binom{w}{v}
=
\binom{w}{u} (n+1)^2.
\]
This, however, is clear.  The lefthand side counts the number of words $v \in
\bW_{m+1}$ such that $u$ is subword of $v$ and $v$ is a subword of $w$.  Any such
$v$ and its embedding in $w$
arises by taking an embedding of $u$ in $w$ and then specifying 
which of the remaining $n+1$ letters $a$ in $w$
and 
which of the remaining $n+1$ letters $b$ in $w$ 
are used to build the word with its particular embedding, and
this is what the righthand side counts.   
\end{proof}

\begin{cor}
The Doob--Martin kernel of $(U_n)_{n \in \bN_0}$ with
distinguished state the empty word is, for $v \in \bW_m$ and $w \in \bW_{m+n}$,
\[
K(v,w) = \binom{w}{v} \frac{\binom{2m}{m}}{{\binom{m+n}{m}}^2}.
\]
\end{cor}

\begin{proof}
We have
\[
\begin{split}
& K(v,w) \\
& \quad = 
\frac{
\PP\{U_{m+n}= w \, \vert \, U_m= v\}
}
{
\PP\{U_{m+n}= w \, \vert \, U_0= \emptyset\}
} \\
& \quad =
\frac
{
\binom{w}{v} \frac{n!n!}{(2m+1) (2m+2) \cdots (2(m+n))}
}
{
\binom{w}{\emptyset} \frac{(m+n)!(m+n)!}{(2(m+n))!}
} \\
& \quad =
\binom{w}{v} 
\frac{n! n! (2(m+n))!}{(m+n)!(m+n)! (2m+1) (2m+2) \cdots (2(m+n))} \\
& \quad =
\binom{w}{v} \frac{\binom{2m}{m}}{\binom{m+n}{n} \binom{m+n}{n}}. \\
\end{split}
\]
\end{proof}

\begin{rem}
\label{identification_D_M_topology}
Up to the factor $\binom{2m}{m}$, the Doob--Martin kernel $K(v,w)$ is the probability that
if we select $m$ of the letters $a$ and $m$ of the letters $b$
uniformly at random from $w$ and list these letters in the same relative
order that they appear in $w$, then the resulting word is $v$.  Therefore,
a sequence  $(w_k)_{k \in \bN}$ in $\bW$ with $N(w_k) \to \infty$ as
$k \to \infty$ converges in the Doob--Martin topology if and only if
for every $m \in \bN$ the sequence of random words in $\bW_m$ obtained by selecting
$m$ letters $a$ and $m$ letters $b$ from $w_k$ (and maintaining their
relative order) converges in distribution as $k \to \infty$.
\end{rem}

\begin{defn}
For $w \in \bW_k$, $k \in \bN_0$, let  
$(U_0^{w}, \ldots, U_k^{w})$ be the bridge from the
empty word to $w$.
\end{defn}

\begin{thm} \label{back}
The backward transition dynamics for all bridges from the
empty word are the
same and consist of removing at each step
one  letter $a$  and one letter $b$  uniformly at random.
\end{thm}

\begin{proof}
Consider the bridge from the empty word to $w \in \bW_k$.

 For $0 \leq m \le k-1$, $v \in \bW_{m+1}$,  and $u \in \bW_{m}$ we have
\begin{equation*}
    \begin{aligned}
    &\PP\{U_m^{w} = u \, | \,  U_{m+1}^{w} = v\} \\
    &= \frac
    {\PP\{U_m =  u, \,  U_{m+1}= v \, | \,  U_k = w\}}
    {\PP\{U_{m+1} = v \, | \, U_k = w\}} \\
    &= \frac
    {\PP\{U_m = u, \, U_{m+1} = v, \,  U_k = w\}}
    {\PP\{U_{m+1} = v, \,  U_k = w\}} \\
    &= \frac
    {\PP\{U_m = u\} 
    \PP\{U_{m+1} = v \, | \, U_m = u\} 
    \PP\{U_k = w \, | U_{m+1} = v\}
    }
    {
    \PP\{U_{m+1} = v\} \PP\{U_k = w \, | U_{m+1} = v\}
    } \\
    & = \binom{v}{u} \frac{\frac{1}{(2m+1)(2m+2)} \times \frac{m!m!}{(2m)!}  }{\frac{(m+1)!(m+1)!}{(2m+2)!}}\\
    &= \frac{\binom{v}{u}} {(m+1)^2}.
    \end{aligned}
\end{equation*}

In order to go backward from the word $v$ of length  $2(m+1)$  to
the word $u$ of length $2m$, we have to remove one  $a$  and one  $b$. 
There are $\binom{v}{u}$ pairs of  $a$  and  $b$   such that
the removal of the pair from $v$ results in $u$,
and there are a total of
$(m+1)^2$ pairs of $a$ and $b$ in $v$, and so the result follows
from the calculation above.
\end{proof}

\section{Labeled infinite bridges}

Suppose that 
$(y_n)_{n\in \bN}$ is a sequence of words in $\bW := \bigsqcup_{n \in \bN_0} \bW_n$
that converges in the Doob--Martin topology and is such that
$N(y_n) \to \infty$ as $n \to \infty$.  
Recall that $(U_0^{y_n}, \ldots, U_{N(y_n)}^{y_n})$, $n \in \bN$, is the associated bridge
that starts from the empty word and is tied to being in state 
$y_n$ at time $N(y_n)$. The finite dimensional
 distributions of $(U_0^{y_n}, \ldots, U_{N(y_n)}^{y_n})$ 
 converge as $n \rightarrow \infty$. Thus, there exists a process 
 $(U_n^\infty)_{n \in \bN_0}$ such that for every $k \in \bN_0$ 
 the random $(k+1)$-tuple $(U_0^{y_n}, \ldots, U_k^{y_n}$) 
 converges in distribution to $(U_0^\infty, \ldots, U_k^\infty)$.

The forward evolution dynamics of the Markov chain 
$(U_n^\infty)_{n \in \bN}$ depend on the sequence $(y_n)_{n \in \bN}$, 
whereas from Section~\ref{DM_background} and Theorem~\ref{back} the backward evolution is Markovian 
and doesn't depend on the sequence $(y_n)_{n \in \bN}$; given 
$U_{k+1}^\infty$, the word $U_k^\infty$ is obtained by removing one 
letter $a$ and one letter $b$ uniformly at random from $U_{k+1}^\infty$.

For each $n \in \bN_0$ the distribution of
$U_n^\infty$ defines the distribution of a random element $\tilde U_{n,n}^\infty$ of 
the set $\tilde \bW_n$ of words of length $2n$ drawn 
from the alphabet $\{a_1, b_1, \ldots, a_n, b_n\}$ with each letter appearing once
by assigning the labels $[n]$ uniformly at random to the letters $a$ and to the letters $b$.
More precisely, for $U_n^\infty = c_1 \ldots c_{2n}$, let $A_n := \{i \in [n] : c_i = a\}$
and $B_n := \{j \in [n] : c_j = b\}$, let $\Sigma: A_n \to [n]$
and $\Tau: B_n \to [n]$ be random bijections that are conditionally independent
and uniformly distributed given $U_n^\infty$, and define
$\tilde U_{n,n}^\infty := \tilde c_1 \ldots \tilde c_{2n}$ by
\[
\tilde c_k :=
\begin{cases}
a_{\Sigma(k)},& k \in A_n,\\
b_{\Tau(k)},& k \in B_n.
\end{cases}
\]

For $0 \le p \le n$, define $\tilde U_{n,p}^\infty$ to be the word obtained by deleting
$\{a_{p+1}, b_{p+1}, \ldots, a_n, b_n\}$ from $\tilde U_{n,n}^\infty$.
Observe that if $0 \le p \le m \wedge n$, then $\tilde U_{m,p}^\infty$, $\tilde U_{n,p}^\infty$ and
$\tilde U_{p,p}^\infty$ have the same distribution.  Moreover, if for $0 \le p \le n$ we let
$U_{n,p}^\infty$ be the result of removing the labels from $\tilde U_{n,p}^\infty$ (that is,
$U_{n,p}^\infty$ is the element of $\bW_p$ obtained by replacing the letters $a_k$, $1 \le k \le p$,
by the letter $a$ and the letters $b_k$, $1 \le k \le p$, by $b$), then
$(U_{n,0}^\infty, \ldots, U_{n,n}^\infty)$ has the same distribution as 
$(U_0^\infty, \ldots, U_n^\infty)$.

By Kolmogorov's consistency theorem, there is a process $(\tilde U_n^\infty)_{n \in \bN_0}$
such that $(\tilde U_0^\infty, \ldots, \tilde U_m^\infty)$ has the same distribution as
$(\tilde U_{n,0}^\infty, \ldots, \tilde U_{n,m}^\infty)$ for any $m \le n$ and
the result of removing the labels from $(\tilde U_n^\infty)_{n \in \bN_0}$
has the same distribution as $(U_n^\infty)_{n \in \bN_0}$.  
By the transfer theorem \cite[Theorem 6.10]{MR1876169},
we may even suppose that $(\tilde U_n^\infty)_{n \in \bN_0}$ is defined
on an extension of the probability space on which $(U_n^\infty)_{n \in \bN_0}$ is
defined in such a way that $(U_n^\infty)_{n \in \bN_0}$ is the result of
removing the labels from $(\tilde U_n^\infty)_{n \in \bN_0}$.

\section{The exchangeable random total order associated with an infinite bridge}
\label{comp}

A state of a labeled infinite bridge is a word of length $2n$ from the
alphabet $\{a_1, b_1, \ldots, a_n, b_n\}$ in which
each letter appears once.  Another way to think of such an object is
as a total order on the set $\bigcup_{k=1}^n\{a_k, b_k\}$.  Because the
labeled infinite bridge evolves by slotting in the letters $a_{n+1}$ and $b_{n+1}$ 
at the $(n+1)^{\mathrm{th}}$ step while leaving
the relative positions of $\{a_1, b_1, \ldots, a_n, b_n\}$ unchanged,
these successive total orders are consistent: the total order 
on $\{a_1, b_1, \ldots, a_n, b_n\}$ given by the state of the
infinite bridge at step $n$ is the same as the total order obtained by taking the
state of the infinite bridge at step $n+1$ (a total order
on $\{a_1, b_1, \ldots, a_n, b_n, a_{n+1}, b_{n+1}\}$)
and looking at the corresponding
induced total order on $\{a_1, b_1, \ldots, a_n, b_n\}$.

This projective structure means that we can associate 
any path of a labeled infinite bridge with a unique total order on
$\bI_0 := \bigcup_{n \in \bN} \{a_n,b_n\}$ such that the induced total
order on $\{a_1, b_1, \ldots, a_n, b_n\}$ coincides with the
state of the labeled infinite bridge at step $n$. 

We now introduce some general notions about random total orders. 

\begin{defn}
A {\em random total order}  $\prec$ on $\bI_0$
is a map from the underlying probability space to the collection
of total orders on $\bI_0$ such that the indicator
$\ind\{x \prec y\}$ is a random variable for every $x,y \in \bI_0$.
A random total order $\prec$
is {\em exchangeable} if for every $ n\in \bN$ the
induced total order $\prec^n$ on $\bigcup_{k=1}^n\{a_k, b_k\}$
has the same distribution as the random total order $\prec^n_{\sigma,\tau}$
 for any permutations $\sigma,\tau$ of $\{1,2,\ldots,n\}$,
where $\prec^n_{\sigma,\tau}$ is defined as follows:
\begin{itemize}
\item $a_{\sigma(i)}\prec^n_{\sigma,\tau}b_{\tau(j)}$ iff $a_i \prec^n b_j$
\item $b_{\tau(i)}\prec^n_{\sigma,\tau}a_{\sigma(j)}$ iff $b_i \prec^n a_j$
\item $a_{\sigma(i)}\prec^n_{\sigma,\tau}a_{\sigma(j)}$ iff $a_i \prec^n a_j$
 \item $b_{\tau(i)}\prec^n_{\sigma,\tau}b_{\tau(j)}$ iff $b_i \prec^n b_j$.
\end{itemize}
\end{defn}

\begin{rem}
The distribution of a random total order $\prec$ is determined by the joint distribution of the
random variables $\{\ind\{x \prec y\} : x,y \in \bigcup_{k=1}^n \{a_k,b_k\}\}$ 
for arbitrary $n \in \bN$.
\end{rem}

\begin{rem}
If $\prec$ is an exchangeable random total order, then
the induced random total orders $\prec^n$, $n \in \bN$, 
are consistent in the sense that
if we take the random total order $\prec^{n+1}$
on $\bigcup_{k=1}^{n+1}\{a_k,b_k\}$ and remove $\{a_{n+1}, b_{n+1}\}$, 
then the induced random total order on $\bigcup_{k=1}^{n}\{a_k,b_k\}$ 
is $\prec^n$.

Conversely, suppose for each $n \in \bN$ that there is a random
total order $\prec^n$ on $\bigcup_{k=1}^n\{a_k, b_k\}$, these
random total orders have the property that $\prec^n$
has the same distribution as $\prec^n_{\sigma,\tau}$
for any permutations $\sigma,\tau$ of $[n]$
for all $n \in \bN$, and these total orders are consistent.
Then there is an exchangeable random order $\prec$ on 
$\bI_0$ such that
$\prec^n$ is the corresponding induced total order on 
$\bigcup_{k=1}^n\{a_k, b_k\}$.
\end{rem}

In terms of these general notions, 
if we let $\prec^n$, $n \in \bN$, be the random total order on $\bigcup_{k=1}^n\{a_k,b_k\}$
corresponding to $\tilde U_n^\infty$, then these total orders are consistent and
there is an exchangeable random total order $\prec$ on $\bI_0$
such that the restriction of $\prec$ to $\bigcup_{k=1}^n\{a_k,b_k\}$ is $\prec^n$.

\section{Characterization of exchangeable random total orders}

The results of the previous sections indicate that if we want to understand
the Doob--Martin compactification, then we need to understand
infinite bridges, and this boils down to understanding
exchangeable random total orders on $\bI_0$.

A mixture of two exchangeable random total orders is also an exchangeable random total order, so we
are interested in exchangeable random total orders $\prec$ that are extremal in the sense that
their distributions cannot be written as a nontrivial mixture of the distributions
of two other exchangeable random total orders.
This is equivalent to requiring that if $A$ is a measurable subset of the space of total orders
on $\bI_0$ with the property that $\prec \in A$ if and only if $\prec^{\sigma,\tau} \in A$
for all finite permutations $\sigma, \tau$, then $\bP\{\prec \in A\} \in \{0,1\}$.  We say that
an exchangeable random total order with this property is {\em ergodic}.

The following result can be established using essentially the same argument as in
Proposition 5.19 (see also the subsequent Remark 5.20) of 
\cite{Remy}, and we omit the details.

\begin{lem}
The tail $\sigma$-field of an infinite bridge $(U_n^\infty)_{n \in \bN_0}$ is almost surely trivial
if and only if the exchangeable random total order induced by the corresponding
labeled infinite bridge $(\tilde U_n^\infty)_{n \in \bN_0}$ is ergodic.
\end{lem}

\begin{rem}
\label{producing_eerto}
There is one obvious way to produce an ergodic exchangeable random total order.  Let $\zeta$ and
$\eta$ be two diffuse probability measures on $\bR$.
Let $(V_n)_{n \in \bN}$ be i.i.d. with common distribution $\zeta$,
let $(W_n)_{n \in \bN}$ be i.i.d. with common distribution $\eta$,
and suppose that these two sequences are independent.
The total order $\prec$ on $\bI_0$ 
defined by declaring
that
\begin{itemize}
\item
$a_i \prec a_j$ if $V_i < V_j$,
\item
$b_i \prec b_j$ if $W_i < W_j$,
\item
$a_i \prec b_j$ if $V_i < W_j$,
\item
$b_i \prec a_j$ if $W_i < V_j$,
\end{itemize}
is exchangeable and ergodic;  exchangeability is obvious and ergodicity
is immediate from the Hewitt--Savage zero--one law applied to the i.i.d.
sequence $((V_n,W_n))_{n \in \bN}$ (indeed, it follows from the
Hewitt--Savage zero--one law that if $A$ is a measurable subset of the space of total orders
on $\bI_0$ with the property that $\prec \in A$ if and only if $\prec^{\rho, \rho} \in A$
for all finite permutations $\rho$, then $\bP\{\prec \in A\} \in \{0,1\}$).

We will show that
all ergodic exchangeable random total orders arise this way.
Note that many pairs of probability measures can give rise to
random total orders with the same distribution: replacing
$\zeta$ and $\eta$ by their push-forwards by some
common strictly increasing function does not change
the distribution of the resulting random total order.
\end{rem}

\begin{defn} 
\label{metric} 
Given an exchangeable random total order $\prec$ on $\bI_0$,
define 
$d: \bI_0 \times \bI_0 \to [0,1]$ 
by requiring that $d(x,x) = 0$ for all $x \in \bI_0$,
$d(x,y) = d(y,x)$ for all $x,y \in \bI_0$, and
\[
\begin{split}
d(x,y)
& := 
\limsup_{n\to \infty} \frac{1}{2n}\#\{1 \le k \le n: x \prec a_k \prec y\} \\
& \quad +
\limsup_{n\to \infty}\frac{1}{2n} \#\{1 \le \ell \le n : x \prec b_\ell \prec y\} \\
\end{split}
\]
for $x \prec y$.  It follows from exchangeability, de Finetti's theorem, and the strong law of large
numbers that in the above the superior limits are actually limits almost surely.  
\end{defn}

\begin{rem}
It is clear that by redefining $d$ on a $\bP$-null set we may assume for every $ x,y,z \in \bI_0 $ that
\begin{itemize}
\item $d(x,y)\geq 0$,
\item $d(x,y)=d(y,x)$,
\item $d(x,z) \le d(x,y)+d(y,z)$,
\item $d(x,y)=0$ if $x=y$.
\end{itemize}
\end{rem}

\begin{rem}
\label{distance_additive}
For distinct $x,y,z \in \bI_0 $
the triangle inequality $d(x,z) \le d(x,y)+d(y,z)$ can be sharpened 
to a statement that for all $x,y,z$
\begin{itemize}
\item 
$d(x,z) = d(x,y) + d(y,z)$ if $x \prec y \prec z$,
\item
$d(x,z) = d(x,y) - d(y,z)$ if $x \prec z \prec y$,
\item
$d(x,z) = d(y,z) - d(x,y)$ if $y \prec x \prec z$,
\end{itemize}
and three analogous equalities when $z \prec x$.
\end{rem}

\begin{prop} 
\label{proofmetric}
If $x,y \in \bI_0$ with $x \ne y$, then
$d(x,y) > 0$ almost surely.  Therefore almost surely $d$ is a metric.
\end{prop}

\begin{proof} 
We need to show for $k,\ell \in \bN$ with $k \ne \ell$ that
$d(a_k, a_\ell) > 0$ and $d(b_k, b_\ell) > 0$, and, furthermore, for arbitrary
$k,\ell \in \bN$ that $d(a_k, b_\ell) > 0$.

Consider $d(a_k, a_\ell)$.
Set
\[
I_m := 
\ind(\{a_k \prec a_m \prec a_\ell\} \cup \{a_\ell \prec a_m \prec a_k\}), \quad m \notin\{k,\ell\}.
\]
Suppose that $\Pi_n$, $n \in \bN$ is a uniform random permutation of
$[n]$. 
By exchangeability of the total order, if $k \vee \ell \le n$, then
\[
\begin{split}
& \bP\{I_m = 0, \, 1 \le m \le n, \, m \notin\{k,\ell\}\} \\
& \quad =
\bP(\{\Pi_n(\ell) = \Pi_n(k) + 1\} \cup \{\Pi_n(k) = \Pi_n(\ell) + 1\}) \\
& \quad = 
2 (n-1) \frac{1}{n(n-1)} \\
& \quad = \frac{2}{n} \\
\end{split}
\]
and the random variables $\{I_m: m \in \bN, \, m \notin\{k,\ell\}\}$ are exchangeable.
It follows from de Finetti's theorem and the strong law of large numbers that
\[
\lim_{n \to \infty} \frac{1}{n} \#\{1 \le m \le n : a_k \prec a_m \prec a_\ell\}
=
\lim_{n \to \infty} \frac{1}{n} \sum_{m=1}^n I_m 
> 0
\]
almost surely and hence $d(a_k, a_\ell) > 0$.
A similar argument shows that $d(b_k, b_\ell) > 0$.

It remains to show that $d(a_k, b_\ell) > 0$.
Set $M := \{m \in \bN : a_k \prec b_m\}$.  It follows
from exchangeability that on the event 
$\{M \ne \emptyset\} \supseteq \{a_k \prec b_\ell\}$ we have $\# M = \infty$ almost surely
and indeed that $\lim_{n \to \infty} \frac{1}{n} \# (M \cap [n]) > 0$.
Write $M = \{m_1, m_2, \ldots\}$ with $m_1 < m_2 < \ldots$.   Fix $p \in \bN$ and set
\[
J_q := \ind\{b_{m_q} \prec b_{m_p}\}, \quad q \ne p.
\]
By exchangeability of the total order, if $p \vee q \le r$, then
\[
\bP\{J_q = 0, \, 1 \le q \le r, \, q \ne p \; | \; M \ne \emptyset\}
=
\bP\{\Pi_r(p) = 1\}
= 
\frac{1}{r}
\]
and the random variables $\{J_q: q \in \bN, \, q \ne p\}$ are conditionally exchangeable given 
$\{M \ne \emptyset\}$.
It follows from de Finetti's theorem that on the event $\{M \ne \emptyset\}$
\[
\lim_{n \to \infty} \frac{1}{n} \# \{q : m_q \in [n], \, a_k \prec b_{m_q} \prec b_{m_p}\} > 0
\]
almost surely and hence $d(a_k, b_\ell) > 0$ almost surely on the event $\{a_k \prec b_\ell\}$.
A similar argument shows that $d(a_k, b_\ell) > 0$ almost surely on the event $\{b_\ell \prec a_k\}$.
\end{proof}

\begin{defn}\label{comple}
Given an ergodic exchangeable random total order $\prec$ on
$\bI_0$, denote by $\bI$ the completion of $\bI_0$ with respect to the metric $d$.
\end{defn}

\begin{defn}
Define $f: \bI_0 \to [0,1]$ by 
\[
f(y) := \sup\{d(x,y) : x \in \bI_0, \, x \prec y\}.
\]
\end{defn}

\begin{rem}
\label{d_f_connection}
It follows from Remark~\ref{distance_additive} that 
\[
\begin{split}
f(y) & = \limsup_{n\to \infty} \frac{1}{2n}\#\{1 \le k \le n: a_k \prec y\} \\
& \quad +
\limsup_{n\to \infty}\frac{1}{2n} \#\{1 \le \ell \le n : b_\ell \prec y\}, \\
\end{split}
\]
\[
|f(x)-f(y)|  = d(x,y), \quad x,y\in \bI_0,
\] 
and
\[
f(x) < f(y) \Longleftrightarrow x \prec y, \quad x,y\in \bI_0,
\]
so that $f$ is an order-preserving isometry from $\bI_0$ into
$[0,1]$.  
Thus the function $f$ extends by continuity to an isometry from
$\bI$ into $[0,1]$ and if $\prec$ is extended to $\bI$ by declaring
that $x \prec y \Longleftrightarrow f(x) < f(y)$, then
$\prec$ is a total order on $\bI$ and $f$ is an order-preserving isometry 
from $\bI$ into $[0,1]$ and hence an order-preserving isometric bijection
from $\bI$ to the image set $\bJ:= f(\bI) \subseteq [0,1]$. Because $\bI$ is
complete, $\bJ$ is complete.  Because $\bJ$ is a complete subset of $[0,1]$
it is closed and hence compact, and therefore $\bI$ itself is compact.
It follows from the ergodicity of $\prec$ that $\bJ$ is almost surely constant.
We will see below that $\bJ = [0,1]$.
\end{rem}

\begin{rem}
\label{def_X_n_Y_n}
Define a sequence $((X_n, Y_n))_{n \in \bN}$ of $\bJ^2$-valued random variables 
by setting $X_n := f(a_n)$ and $Y_n := f(b_n)$.  The exchangeability of $\prec$ implies 
that if $\sigma$ and $\tau$ are two finite permutations
of $\bN$, then $((X_{\sigma(n)}, Y_{\tau(n)}))_{n \in \bN}$ has the same
distribution as $((X_n, Y_n))_{n \in \bN}$.  In particular, the sequence $((X_n, Y_n))_{n \in \bN}$ is exchangeable.  
It is a consequence of
de Finetti's theorem and the ergodicity of $\prec$ that this sequence is i.i.d. with common distribution
some probability measure $\pi$ on $\bJ^2$.  It follows from the next result that
$\pi = \mu \otimes \nu$ for two probability measures $\mu$ and $\nu$ on $\bJ$ that we call the {\em canonical pair}. Because $X_m \ne X_n$ and
$Y_m \ne Y_n$ almost surely for $m \ne n$, the probability measures $\mu$ and $\nu$ must be diffuse.
\end{rem}

\begin{lem}\label{ind2}
Suppose that the random variables $X',Y',X'',Y''$ are such that
\begin{enumerate}
\item $(X',Y')\,{\buildrel d \over =}\,  (X'',Y'')$
\item $((X',Y'),(X'',Y''))\,{\buildrel d \over =}\, ((X',Y''),(X'',Y'))$
\item $(X',Y')\independent (X'',Y'')$.
\end{enumerate}
Then $X',X'',Y',Y''$ are independent.
\end{lem}

\begin{proof}
For Borel sets $A', A'', B', B''$ we have
\begin{equation*}
\begin{aligned}
&\PP \{X' \in A',\, X''\in A'',\, Y' \in B',\, Y'' \in B''\}& \\
&= \PP \{X' \in A',\, Y'\in B'\}\PP\{ X'' \in A'',\, Y'' \in B''\} & \text{by (3)}\\
&= \PP\{X' \in A',\, Y''\in B'\}\PP\{ X'' \in A'',\, Y' \in B''\} & \text{by (2)}\\
&= \PP\{X' \in A'\}\PP\{Y''\in B'\}\PP\{X'' \in A''\}\PP\{Y' \in B''\} & \text{by (3)}\\
&= \PP\{X' \in A'\}\PP\{Y'\in B'\}\PP\{X'' \in A''\} \PP\{Y'' \in B''\}& \text{by (1)}.
\end{aligned}
\end{equation*}
\end{proof}

\begin{thm}
Any ergodic exchangeable random total order $\prec$ has the same distribution as one given by the
construction in Remark~\ref{producing_eerto} for some pair of diffuse probability measures
$(\zeta, \eta)$ on $\bR$.  The canonical pair of diffuse probability measures $(\mu, \nu)$ on $[0,1]$
is uniquely determined by the moment formulae
\[
\begin{split}
& \int_{[0,1]} x^n \, \mu(dx) \\
& \quad =
\left(\frac{1}{2}\right)^n \sum_{c \in \prod_{k=1}^n \{a_k, b_k\}} \bP\{c_1 \prec a_{n+1}, \, \ldots, \, c_n \prec a_{n+1}\} \\
\end{split}
\] 
and
\[
\begin{split}
& \int_{[0,1]} y^n \, \nu(dy) \\
& \quad =
\left(\frac{1}{2}\right)^n \sum_{c \in \prod_{k=1}^n \{a_k, b_k\}} \bP\{c_1 \prec b_{n+1}, \, \ldots, \, c_n \prec b_{n+1}\}. \\
\end{split}
\]
The probability measure $\frac{1}{2}(\mu + \nu)$ is Lebesgue measure on $[0,1]$ and, in particular, $\bJ = [0,1]$.
Moreover, $\mu$ and $\nu$ are the respective push-forwards
of $\zeta$ and $\eta$ by the function $z \mapsto \frac{1}{2}(\zeta + \eta)((-\infty,z])$ 
\end{thm}

\begin{proof}
We have already shown that an
ergodic exchangeable random total order has the same distribution as one built from an arbitrary
pair $(\zeta, \eta)$ of diffuse probability measures on $\bR$ using the
construction in Remark~\ref{producing_eerto}.

Define $((X_n,Y_n))_{n \in \bN}$ as in Remark~\ref{def_X_n_Y_n}.
It follows from Remark~\ref{d_f_connection} that 
\[
X_n = \frac{1}{2} \mu((-\infty,X_n]) + \frac{1}{2} \nu((-\infty,X_n])
\]
and
\[
Y_n = \frac{1}{2} \mu((-\infty,Y_n]) + \frac{1}{2} \nu((-\infty,Y_n])
\]
for any $n \in \bN$.  Let $(I_n)_{n \in \bN}$ be a sequence of i.i.d. random variables that
is independent of $((X_n, Y_n))_{n \in \bN}$ with $\bP\{I_n = 0\} = \bP\{I_n = 1\} = \frac{1}{2}$
and set $Z_n := I_n X_n + (1 - I_n) Y_n$ so that the sequence $(Z_n)_{n \in \bN}$ is i.i.d. with common
distribution $\frac{1}{2}(\mu + \nu)$.  We have
\[
Z_n = \frac{1}{2}(\mu + \nu)((-\infty, Z_n]),
\]
and so $\frac{1}{2}(\mu + \nu)$ is Lebesgue measure on $[0,1]$.
Thus, for any $n \in \bN$ 
\[
\begin{split}
\int_{[0,1]} x^n \, \mu(dx)
& =
\bP\{Z_1<X_{n+1}, \, \ldots, \, Z_n<X_{n+1}\} \\
& =
\left(\frac{1}{2}\right)^n \sum_{c \in \prod_{k=1}^n \{a_k, b_k\}} \bP\{c_1 \prec a_{n+1}, \, \ldots, \, c_n \prec a_{n+1}\} \\
\end{split}
\] 
and
\[
\begin{split}
\int_{[0,1]} y^n \, \nu(dy)
& =
\bP\{Z_1<Y_{n+1}, \, \ldots, \, Z_n<Y_{n+1}\} \\
& =
\left(\frac{1}{2}\right)^n \sum_{c \in \prod_{k=1}^n \{a_k, b_k\}} \bP\{c_1 \prec b_{n+1}, \, \ldots, \, c_n \prec b_{n+1}\}, \\
\end{split}
\]
as claimed.

The proof of the final claim is straightforward and we omit it.
\end{proof}

\begin{rem}
We haven't shown that if $(y_k)_{k \in \bN}$ is a sequence of points of $\bW$, where  $y_k \in \bW_{N(y_k)}$,
$N(y_k) \to \infty$ as $k \to \infty$, and 
$\lim_{k \to \infty} y_k = y$ in the Doob--Martin topology for some arbitrary $y$ in the Doob--Martin boundary,
then the harmonic function $K(\cdot,y)$ is extremal.   This is equivalent to showing that if the infinite bridge $(U_n^\infty)_{n \in \bN_0}$
is the limit of the bridges $(U_0^{y_k}, \ldots, U_{N(y_k)}^{y_k})$, then $(U_n^\infty)_{n \in \bN_0}$ has an almost surely trivial
tail $\sigma$-field.  This is, in turn, equivalent to showing that the corresponding labeled infinite bridge induces an ergodic
exchangeable random order.  The latter, however, can be established along the lines of \cite[Corollary~5.21]{Remy} and \cite[Corollary~7.2]{radix},
so we omit the details.
\end{rem}

\section{Identification of extremal harmonic functions}

Any extremal infinite bridge $(U_n^\infty)_{n \in \bN_0}$
is the $h$-transform
of our original Markov chain with an extreme harmonic function $h$.
We know from the above that such a process arises as follows
in terms of the canonical pair $(\mu, \nu)$ of diffuse probability measures
associated with the corresponding point in the Doob--Martin boundary.

We first require some notation.
Given $(x_1, \ldots, x_n, y_1, \ldots, y_n) \in \bR^{2n}$ with distinct
entries, let $z_1 < \cdots < z_{2n}$ be a listing of
$\{x_1, \ldots, x_n, y_1, \ldots, y_n\}$ in increasing order.
Define 
\[
\cW((x_1, \ldots, x_n, y_1, \ldots, y_n)) = u_1 \ldots u_{2n} \in \bW_n
\] 
by
\[
u_i = 
\begin{cases}
a,& \quad \text{if $z_i \in \{x_1, \ldots, x_n\}$}, \\
b,& \quad \text{if $z_i \in \{y_1, \ldots, y_n\}$}. \\
\end{cases}
\]
Given $v \in \bW_n$, set
\[
\cS(v) := \cW^{-1}(\{v\}) \subset \bR^{2n}.
\]
For example,
\[
\cS(abba) = \bigsqcup_{\sigma, \tau} \{(x_1,x_2,y_1,y_2) \in \bR^4: x_{\sigma(1)} < y_{\tau(1)} < y_{\tau(2)} < x_{\sigma(2)}\},
\]
where the union is over all pairs of permutations $\sigma, \tau$ of the set $\{1,2\}$.  In general, $\cS(v)$ is the disjoint union
of $(n!)^2$ connected open sets that all have boundaries of zero Lebesgue measure.

Now take independent
sequences of real-valued random variables $(X_k)_{k \in \bN}$ and
$(Y_k)_{k \in \bN}$, where the $X_k$ are i.i.d. with common distribution $\mu$
and the $Y_k$ are i.i.d. with common distribution $\nu$
and set
\[
U_n^\infty
=
\cW((X_1, \ldots, X_n, Y_1, \ldots, Y_n)).
\]

We have
\[
\bP\{U_n^\infty = u\} 
= 
\mu^{\otimes n} \otimes \nu^{\otimes n}(\cS(u))
\]
We also know that
\[
\bP\{U_n^\infty = u \, | \, U_{n+1}^\infty = v\}
=
\frac{\binom{v}{u}}{(n+1)^2}.
\]
It follows that
\[
\begin{split}
& \bP\{U_{n+1}^\infty = v \, | \, U_n^\infty = u\} \\
& \quad =
\mu^{\otimes (n+1)} \otimes \nu^{\otimes (n+1)}(\cS(v)) \frac{\binom{v}{u}}{(n+1)^2}
\bigg / 
\mu^{\otimes n} \otimes \nu^{\otimes n}(\cS(u)).\\
\end{split}
\]
On the other hand,
\[
\begin{split}
\bP\{U_{n+1}^\infty = v \, | \, U_n^\infty = u\}
& =
\frac{1}{h(u)} \bP\{U_{n+1} = v \, | \, U_n = u\} h(v) \\ 
& = 
\frac{h(v)}{h(u)}
\frac{\binom{v}{u}}{(2n+2)(2n+1)}.\\
\end{split}
\]
Thus,
\[
\frac{h(v)}{h(u)}
=
\frac
{\mu^{\otimes (n+1)} \otimes \nu^{\otimes (n+1)}(\cS(v))}
{\mu^{\otimes n} \otimes \nu^{\otimes n}(\cS(u))}
\frac{(2n+2)(2n+1)}{(n+1)^2}
\]
and, up to an arbitrary multiplicative constant,
\[
h(w) = 
\binom{2m}{m}
\mu^{\otimes m} \otimes \nu^{\otimes m}(\cS(w))
\]
for $w \in \bW_m$.

Since $h(\emptyset) = 1$, this normalization is the extended Doob--Martin kernel
$w \mapsto K(w,y)$, where $y$ is the point in the Doob--Martin boundary that corresponds
to the pair of diffuse probability measures $(\mu, \nu)$.

\begin{rem}
The constant harmonic function $h \equiv 1$ arises from the above construction
with $\mu$ and $\nu$ both being the Lebesgue measure $\lambda$ on $[0,1]$.  Therefore the
process $(U_n)_{n \in \bN_0}$ is itself the extremal bridge associated with the canonical
pair $(\lambda, \lambda)$.  In particular, $(U_n)_{n \in \bN_0}$ converges almost surely
to the point in the Doob--Martin boundary associated with this pair.
\end{rem}

%It follows that if $(y_k)_{k \in \bN}$ is a sequence of points
%in the Doob--Martin boundary such that $y_k$ corresponds to the pair of diffuse probability measures $\mu_k, \nu_k$, then
%$\lim_{k \to \infty} y_k = y$ if and only if 
%$\lim_{k \to \infty} \mu_k^{\otimes m} \otimes \nu_k^{\otimes m}(\cS(w)) = \mu^{\otimes m} \otimes \nu^{\otimes m}(\cS(w))$
%for all $w \in \bW_m$ for all $m \in \bN$.  

We observed in Remark~\ref{identification_D_M_topology}
that a sequence  $(y_k)_{k \in \bN}$ with $N(y_k) \to \infty$
as $k \to \infty$ converges in the Doob--Martin topology 
 if and only if
for every $m \in \bN$ the sequence of random words in $\bW_m$ obtained by selecting
$m$ letters $a$ and $m$ letters $b$ uniformly at random from $y_k$ and maintaining their
relative order converges in distribution as $k \to \infty$.  We can now enhance
that result as follows.

\begin{prop}
\label{convergence_criterion}
Consider a sequence $(y_k)_{k \in \bN}$ in $\bW$, where  $y_k \in \bW_{N(y_k)}$, $k \in \bN$,
and $N(y_k) \to \infty$ as $k \to \infty$.
If $y$ is the point in the Doob--Martin boundary that corresponds
to the pair of (diffuse) probability measures $(\mu, \nu)$ with $\frac{1}{2}(\mu+\nu) = \lambda$, 
then $\lim_{k \to \infty} y_k = y$ in the Doob--Martin topology if and only if
\[
\lim_{k \to \infty} \frac{\binom{y_k}{w}} {{\binom{N(y_k)}{m}}^2}
=
\mu^{\otimes m} \otimes \nu^{\otimes m}(\cS(w))
\]
for all $w \in \bW_m$ for all $m \in \bN$.  That is, $\lim_{k \to \infty} y_k = y$
if and only if for each $m \in \bN$ the sequence of random words in $\bW_m$ obtained
by selecting
$m$ letters $a$ and $m$ letters $b$ uniformly at random from $y_k$ and maintaining their
relative order converges in distribution as $k \to \infty$ to the random word
$U_m^\infty = \cW(X_1, \ldots, X_m, Y_1, \ldots, Y_m)$ defined above.
\end{prop}

%For example, when $m=1$ we have $\bW_1 = \{ab, ba\}$,
%\[
%\binom{v}{ab} = \#\{1 \le i < j \le 2 N(v) : v(i) = a, \, v(j) = b\},
%\]
%and
%\[
%\binom{v}{ba} = \#\{1 \le i < j \le 2 N(v) : v(i) = b, \, v(j) = a\},
%\]
%so we require that
%\[
%\begin{split}
%& \lim_{k \to \infty} \frac{1}{N(v_k)^2}\#\{1 \le i < j \le 2 N(v_k) : v_k(i) = a, \, v_k(j) = b\} \\
%& \quad =
%\mu \otimes \nu(\{(s,t) : s < t\}) \\
%\end{split}
%\]
%and
%\[
%\begin{split}
%& \lim_{k \to \infty} \frac{1}{N(v_k)^2}\#\{1 \le i < j \le 2 N(v_k) : v_k(i) = b, \, v_k(j) = a\} \\
%& \quad =
%\mu \otimes \nu(\{(s,t) : s > t\}).\\
%\end{split}
%\]

Given a sequence $(y_k)_{k \in \bN}$ in $\bW$, where  $y_k \in \bW_{N(y_k)}$, $k \in \bN$,
and $N(y_k) \to \infty$ as $k \to \infty$, define a sequence of pairs of discrete probability
measures $((\mu_k, \nu_k))_{k \in \bN}$ on $[0,1]$ as follows.  For $k \in \bN$ the two
probability measure $\mu_k$ and $\nu_k$ both assign all of their mass to the set $\{\frac{\ell}{2 N(y_k)} : 1 \le \ell \le 2 N(y_k)\}$.
For $1 \le i \le 2 N(y_k)$, $\mu_k(\frac{i}{2 N(y_k)}) = \frac{1}{N(y_k)}$ if the $i^{\mathrm{th}}$ letter of $y_k$ is the letter $a$, otherwise
$\mu_k(\frac{i}{2 N(y_k)}) = 0$.  Similarly, for $1 \le j \le 2 N(y_k)$, $\nu_k(\frac{j}{2 N(y_k)}) = \frac{1}{N(y_k)}$
 if the $j^{\mathrm{th}}$ letter of $y_k$ is the letter $b$, otherwise
$\nu_k(\frac{j}{2 N(y_k)}) = 0$.  In particular, $\frac{1}{2}(\mu_k + \nu_k)$ is the uniform probability measure
on $\{\frac{\ell}{2 N(y_k)} : 1 \le \ell \le 2 N(y_k)\}$.  Observe that if $w \in \bW_m$, then, for $w \in \bW_m$,
\[
(N(y_k)^m)^2 \mu_k^{\otimes m} \otimes \nu_k^{\otimes m}(\cS(w))
=
(m!)^2 \binom{y_k}{w}
\]
so that
\[
\frac{\binom{y_k}{w}} {{\binom{N(y_k)}{m}}^2} = \left(\frac{N(y_k)^m}{N(y_k) (N(y_k)-1) \cdots (N(y_k) - m + 1)}\right)^2 \mu_k^{\otimes m} \otimes \nu_k^{\otimes m}(\cS(w)).
\]
One direction of the following corollary is now immediate.

\begin{cor}
Suppose that $(y_k)_{k \in \bN}$ and $((\mu_k, \nu_k))_{k \in \bN}$ are as above.  If $(y_k)_{k \in \bN}$ converges in the Doob--Martin topology to
the point $y$ in the Doob--Martin boundary that corresponds to the pair of probability measures $(\mu, \nu)$, then
$(\mu_k)_{k \in \bN}$ converges weakly to $\mu$ and $(\nu_k)_{k \in \bN}$ converges weakly to $\nu$.  Conversely, if 
$(\mu_k)_{k \in \bN}$ converges weakly to $\mu$ and $(\nu_k)_{k \in \bN}$ converges weakly to $\nu$, then $\frac{1}{2}(\mu+\nu) = \lambda$,
and if $y$ is the point in the Doob--Martin boundary
that corresponds to the pair $(\mu, \nu)$, then $(y_k)_{k \in \bN}$ converges in the Doob--Martin topology to $y$.
\end{cor}

\begin{proof}
As we have already remarked, if $(\mu_k)_{k \in \bN}$ converges weakly to $\mu$ and $(\nu_k)_{k \in \bN}$ converges weakly to $\nu$ then, since the boundary of
$\cS(w)$ is Lebesgue null for any word $w \in \bW_m$, $m \in \bN$, we have that $\mu_k^{\otimes m} \otimes \nu_k^{\otimes m}(\cS(w))$ converges to $\mu^{\otimes m} \otimes \nu^{\otimes m}(\cS(w))$
so that 
\[
\lim_{k \to \infty} \frac{\binom{y_k}{w}} {{\binom{N(y_k)}{m}}^2}
=
\mu^{\otimes m} \otimes \nu^{\otimes m}(\cS(w)),
\]
and it follows from Proposition~\ref{convergence_criterion} that $(y_k)_{k \in \bN}$ converges to the point $y$ in the Doob--Martin boundary that corresponds to the pair $(\mu,\nu)$.

Conversely, suppose that $(y_k)_{k \in \bN}$ converges to the point $y$ in the Doob--Martin boundary corresponding to the pair $(\mu,\nu)$.  Given any subsequence of
$\bN$ there is, by the compactness in the weak topology of probability measures on $[0,1]$, a further subsequence such that along this further subsequence
$\mu_k$ converges weakly to some probability measure $\mu'$ and $\nu_k$ converges weakly to some probability measure $\nu'$.  Note that
$\frac{1}{2}(\mu'+\nu') = \lambda$.  From the other direection of the corollary,
this implies that along the subsubsequence $y_k$ converges to the point $y'$ in the Doob--Martin boundary corresponding to the pair of
probability measures $(\mu', \nu')$.  Because $y' = y$ it must be the case $(\mu', \nu') = (\mu, \nu)$.  Thus, from any subsequence of
$\bN$ we can extract a further subsequence along which $\mu_k$ converges weakly to $\mu$ and $\nu_k$ converges weakly to $\nu$, and this
implies that $(\mu_k)_{k \in \bN}$ converges weakly to $\mu$ and $(\nu_k)_{k \in \bN}$ converges weakly to $\nu$.
\end{proof}

\section{Example: the Plackett-Luce chain}

In general, there is no simple closed form expression for the transition probabilities of an
infinite bridge $(U_n^\infty)_{n \in \bN_0}$ associated
with a pair of (not necessarily canonical)
diffuse probability measures $\zeta, \eta$ and hence the associated harmonic function $h$.
However, it is possible to obtain such expressions in
the special case where $\zeta$ is the exponential distribution with
rate parameter $\alpha$ and $\eta$ is the exponential distribution with
rate parameter $\beta$.  Given $u \in \bW_n$ and 
$1 \le i \le 2n$, set
\[
\mathbf{A}_i^n(u) := \#\{i \le j \le 2n : u_j = a\}
\]
and
\[
\mathbf{B}_i^n(u) := \#\{i \le j \le 2n : u_j =  b\}.
\]
By the reasoning that goes into the analysis
of the Plackett-Luce or vase model of random permutations
(see, for example, \cite{MR1346107}),
\[
\bP\{U_n^\infty = u\}
=
(n!)^2 \alpha^n \beta^n
\prod_{i=1}^{2n} 
\frac{1}{\mathbf{A}_i^n(u) \alpha + \mathbf{B}_i^n(u) \beta}
\]
--
this is essentially just repeated applications of the elementary result usually
called {\em competing exponentials}: if $S$ and $T$ are independent 
exponentially distributed random variables with rate parameters
$\lambda$ and $\theta$, then the probability of the event $\{S < T\}$
is $\frac{\lambda}{\lambda + \theta}$ and conditional on this event
the random variables $S$ and $T-S$ are independent and exponentially
distributed with rate parameters $\lambda + \theta$ and $\theta$.
(As a check, note that when $\alpha = \beta = \gamma$, say,
this probability is, as expected, $1/\binom{2n}{n}$.)
We also know that
\[
\bP\{U_n^\infty = u \, | \, U_{n+1}^\infty = v\}
=
\frac{\binom{v}{u}}{(n+1)^2}.
\]

It follows that
\[
\begin{split}
& \bP\{U_{n+1}^\infty = v \, | \, U_n^\infty = u\} \\
& \quad =
\frac{\binom{v}{u}}{(n+1)^2}
((n+1)!)^2 \alpha^{n+1} \beta^{n+1}
\prod_{i=1}^{2(n+1)} 
\frac{1}{\mathbf{A}_i^{n+1}(v) \alpha + \mathbf{B}_i^{n+1}(v) \beta} \\
& \qquad \bigg /
(n!)^2 \alpha^n \beta^n
\prod_{i=1}^{2n} 
\frac{1}{\mathbf{A}_i^n(u) \alpha + \mathbf{B}_i^n(u) \beta} \\
& \quad =
\binom{v}{u}
\alpha \beta 
\frac
{\prod_{i=1}^{2n} (\mathbf{A}_i^n(u) \alpha + \mathbf{B}_i^n(u) \beta)}
{\prod_{i=1}^{2(n+1)} (\mathbf{A}_i^{n+1}(v) \alpha + \mathbf{B}_i^{n+1}(v) \beta)}. \\
\end{split}
\]

As a check, when $\alpha = \beta = \gamma$, say, this transition probability is
\[
\binom{v}{u} \frac{(2n)!}{(2(n+1))!} = \frac{\binom{v}{u}}{(2n+2)(2n+1)},
\]
as expected.

The corresponding harmonic function $h$ satisfies
\[
\begin{split}
& \binom{v}{u}
\alpha \beta 
\frac
{\prod_{i=1}^{2n} (\mathbf{A}_i^n(u) \alpha + \mathbf{B}_i^n(u) \beta)}
{\prod_{i=1}^{2(n+1)} (\mathbf{A}_i^{n+1}(v) \alpha + \mathbf{B}_i^{n+1}(v) \beta)} \\
& \quad =
\frac{h(v)}{h(u)} \frac{\binom{v}{u}}{(2n+2)(2n+1)}. \\
\end{split}
\]
We conclude from this that, up to an arbitrary positive constant,
\[
h(w) = \frac{(2m)! \alpha^m \beta^m}
{\prod_{i=1}^{2m} (\mathbf{A}_i^m(w) \alpha + \mathbf{B}_i^m(w) \beta)}
\]
for $w \in \bW_n$.

\vskip 24pt
\noindent
{\bf Acknowledgments.}  We thank two anonymous referees for a number of helpful suggestions that
improved the presentation considerably.

\def\cprime{$'$} \def\cprime{$'$} \def\cprime{$'$}
\providecommand{\bysame}{\leavevmode\hbox to3em{\hrulefill}\thinspace}
\providecommand{\MR}{\relax\ifhmode\unskip\space\fi MR }
\providecommand{\MRhref}[2]{%
  \href{http://www.ams.org/mathscinet-getitem?mr=#1}{#2}
}
\providecommand{\href}[2]{#2}

\end{document}